\documentclass{amsart}
\usepackage{amssymb}

\def\squareforqed{\hbox{\rlap{$\sqcap$}$\sqcup$}}
\def\qed{\ifmmode\squareforqed\else{\unskip\nobreak\hfil
\penalty50\hskip1em\null\nobreak\hfil\squareforqed
\parfillskip=0pt\finalhyphendemerits=0\endgraf}\fi\medskip}

\newcommand{\udot}{{}^{\textstyle .}}

\newcommand{\Mat}{\mathrm M}

\newcommand{\He}{\mathrm{He}}
\newcommand{\Ja}{\mathrm{J}}

\newcommand{\Aut}{\mathrm{Aut}}

\newtheorem{theorem}{Theorem}
\newtheorem{lemma}{Lemma}

\title{A negative answer to a question of Aschbacher}
\author{Robert A. Wilson}
\address{School of Mathematical Sciences, Queen Mary University of London,
Mile End Road, London E\textup{1 4}NS, U.K.}
   \email{r.a.wilson@qmul.ac.uk}
\date{First draft 11/04/2018; this version 07/06/2018}

\begin{document}
%\maketitle

\begin{abstract}
We give infinitely many examples
%, of
%various types, 
to show that, even for simple groups $G$, it is
possible for the lattice of overgroups of a subgroup $H$ to be
the Boolean lattice of rank $2$, in such a way that the two maximal
overgroups of $H$ are conjugate in $G$. This answers negatively a question
posed by Aschbacher.
\end{abstract}
\maketitle
\section{The question}
In a recent survey article on the subgroup structure of finite groups
\cite{Asch}, in the context of discussing open problems on the possible structures of
subgroup lattices of finite groups,
Aschbacher poses the following specific question. Let $G$ be a finite
group, $H$ a subgroup of $G$, and suppose that $H$ is contained in exactly
two maximal subgroups $M_1$ and $M_2$ of $G$, and %furthermore 
that $H$ is maximal
in both $M_1$ and $M_2$. Does it follow that $M_1$ and $M_2$ are not
conjugate in $G$? This is Question 8.1 in \cite{Asch}.
For $G$ a general group, he asserts there is a counterexample, 
%though he does not give this counterexample 
not given in \cite{Asch},
%though he does not tell us what it is, 
so
he restricts this question to the case $G$ almost simple, that is
$S\le G\le\Aut(S)$ for some simple group $S$. This is Question 8.2 in \cite{Asch}.
\section{The answer}
In fact, the answer is no, even for simple groups $G$. 
%Ironically, in view of the
%title of the conference proceedings in which Aschbacher's paper appears, a 
%number of counterexamples can be found %quite easily 
%in the Atlas of Finite Groups
%\cite{Atlas}, whose notation we follow here. 
The smallest example seems to be the simple Mathieu group
$\Mat_{12}$
of order $95040$.
%the following.
\begin{theorem}
Let $G= \Mat_{12}$, and $H\cong A_5$
acting transitively on the $12$ points permuted by $\Mat_{12}$. 
Then $H$ lies in exactly two other subgroups of $G$, both lying in the
single conjugacy class of maximal subgroups $L_2(11)$.
\end{theorem}
\begin{proof}
The maximal subgroups of $G$ are well-known, and are listed in 
%It is easy to see
%from 
the Atlas of Finite Groups
\cite[p.~33]{Atlas}. From this list it follows 
that the only maximal subgroups of $G$ that contain $H$
are conjugates of the transitive subgroup $M\cong L_2(11)$. 
The maximal subgroups of $\Aut(G)\cong \Mat_{12}{:}2$ are determined in
\cite{maxaut}, where it is shown in particular that the normalizers in
$\Aut(G)$ of $H$ and $M$ are $\Aut(H)\cong S_5$ and $\Aut(M)\cong PGL_2(11)$
respectively. Since $PGL_2(11)$ does not contain $S_5$, it follows
that there are precisely two conjugates of $M$ that contain $H$,
and that these conjugates are interchanged by elements of $\Aut(H)\setminus H$.
%such maximal
%subgroups follows from the results of \cite{maxaut} on maximal subgroups
%of $\Aut(G)\cong \Mat_{12}{:}2$, specifically the fact that the normalizer of $H$ in $\Aut(G)$
%is $\Aut(H)\cong S_5$, which is not a subgroup of $\Aut(L_2(11))\cong L_2(11){:}2$.
\end{proof}
%The facts presented in the above proof 
%are well-known, at least since 1985, when this example was
%presented in \cite{maxaut}, but probably even twenty
%years earlier than that. 
%\end{proof}
%Indeed, this example was used in \cite{maxaut} as a counterexample
%to a question similar to Aschbacher's question above, namely the question
%whether, in the case $S$ simple and $\Aut(S)\cong S.2$, a maximal subgroup of $\Aut(S)$,
%other than $S$,
%is necessarily the normalizer of either a maximal subgroup of $S$, or the
%intersection of two non-conjugate maximal subgroups of $S$. This example
%answered the question in the negative, and thereby
%showed definitively that the maximal subgroup problem for almost simple groups
%cannot be effectively reduced to the same problem for simple groups.
%This %important 
%%fact %needs to be emphasized
%is important, and the role of this critical example
%is re-iterated in the introduction to my survey paper \cite{maxspor} in the
%same conference proceedings.

Of course, one example answers the specific question, but does not
address the context in which the question was asked. One needs to consider
rather how many examples there are, or whether the phenomenon
just exhibited is relatively common or rare.
In fact, 
the results of \cite{maxaut} can be used to deduce the existence of one more
example, %and can be found in
%\cite[p.~104]{Atlas}, 
that is 
in the sporadic simple group of Held. The relevant subgroup information,
obtained in \cite{Butler,maxaut},
is summarised
in \cite[p.~104]{Atlas}.
%the following.
\begin{theorem}
Let $G=\He$ and $H\cong (A_5\times A_5).2.2$.
Then $H$ is contained in just two other subgroups of $G$, both lying in the single
class of maximal subgroups isomorphic to $S_4(4){:}2$.
\end{theorem}
\begin{proof}
It is shown in \cite{Butler} (see also \cite{maxaut}) 
that there is a unique class of $A_5\times A_5$ in the
Held group, and that the normalizer of any $A_5\times A_5$ in $G$ is
a group $H$ of shape $(A_5\times A_5).2^2$, in which there is no normal $A_5$.
It follows that $H$ lies inside a maximal subgroup $M\cong S_4(4){:}2$.
Now in $\Aut(G)$ the normalizers of $H$ and $M$ are $S_5\wr 2$ and $S_4(4){:}4$
respectively. But $S_4(4){:}4$ does not contain $S_5\wr 2$, so the elements
of $\Aut(H)\setminus H$ interchange two $G$-conjugates of $M$ that contain $H$.
\end{proof}
The above examples constitute the extent of general knowledge at the time
of the publication of the Atlas.
 
\section{Doubly-deleted doubly-transitive permutation representations}
Both the examples given so far occur in sporadic groups $G$. There is also at least one
example in which $H$ is sporadic, but $G$ is a classical simple group.
%Lest it be thought that this is a phenomenon that only occurs for sporadic groups $G$,
%and can therefore be safely ignored, 
%I point out the following example in a classical group, in this case a simple 
%orthogonal group of minus type.
\begin{theorem}
Let $G=\Omega_{10}^-(2)$ and
$H\cong \Mat_{12}$. Then $H$ is contained in exactly two other subgroups of $G$,
both lying in the single conjugacy class of subgroups isomorphic to $A_{12}$.
\end{theorem}
% provides another example.
\begin{proof} 
In this case there is a crucial error in
\cite[p.~147]{Atlas} and one needs to use the corrected list of maximal subgroups of $G$
from \cite{ABC} or \cite{BHRD}. Note in particular that
$\Aut(G)\cong \Omega_{10}^-(2){:}2$ contains maximal subgroups $S_{12}$ and 
$\Aut(\Mat_{12})\cong \Mat_{12}{:}2$. Since
 $S_{12}$ does not contain $\Mat_{12}{:}2$, we have essentially the same
situation as in the two previous examples. The only maximal subgroups of $G$ that
contain $H$ are conjugates of $M\cong A_{12}$, and there are exactly two such
conjugates, swapped by elements of $\Aut(H)\setminus H$. 
%Again it is easy to see that the only maximal
%subgroups of $G$ that contain $H$ are isomorphic to $A_{12}$, that there are
%exactly two such copies of $A_{12}$, and that they are conjugate in $G$.
\end{proof}
%\section{An old example?}
%\section{An infinite series of examples}
Analysing this %third 
example, %given above, 
it is clear that an important property
of $\Mat_{12}$ that is being used here is that it has two distinct $2$-transitive
representations on $12$ points, swapped by the outer automorphism. The smallest
simple group with such a property is $L_3(2)$, which has two distinct $2$-transitive
representations on $7$ points. In characteristic $7$, therefore, there is a
doubly-deleted permutation representation, giving rise to an embedding in 
$\Omega_5(7)$.
%By analogy with $\Mat_{12}$, consider
%the embedding $L_3(2)<A_7<\Omega_5(7)$.

%Referring to \cite{BHRD} for detailed information on maximal subgroups
%of orthogonal groups, 
%
%Now the maximal subgroups of $\Omega_5(7)\cong S_4(7)$ were determined by
%H. H. Mitchell \cite{HHM} in 1914. In principle, his results are sufficient also
%to determine the maximal subgroups of $\Aut(S_4(7))\cong S_4(7){:}2$.
%However, he determines only the orders of the maximal subgroups, not their
%structures, and he makes some mistakes, which lead him to conclude,
%incorrectly, that $S_4(7)$
%does not contain a subgroup of order $168$, other than those that
%fix a $2$-space in the natural representation.
\begin{theorem}
Let $G=\Omega_5(7)$ and $H\cong L_3(2)\cong L_2(7)$ be a subgroup of $G$, acting
irreducibly on the $5$-dimensional module. Then $H$ is contained in exactly two
other subgroups of $G$, both isomorphic to $A_7$, and lying in the same
$G$-conjugacy class. 
\end{theorem}
\begin{proof}
Reading off the information about
irreducible subgroups of $\Omega_5(7)$ and $SO_5(7)$ 
%instead 
from \cite[Table 8.23]{BHRD},
we see that $G=\Omega_5(7)$ has 
a single class of irreducible subgroups $H=L_3(2)$, and these subgroups are contained
in maximal subgroups $M=A_7$. Correspondingly, in $SO_5(7)$ there are maximal subgroups
$L_3(2){:}2$ and $S_7$. The outer automorphism of $L_3(2)$ therefore swaps two
($G$-conjugate) copies of $A_7$ containing $L_3(2)$. 
\end{proof}

More generally, for all $n\ge 3$ and all prime powers $q$, the simple group
$L_n(q)$ has two inequivalent $2$-transitive permutation
representations on $d:=(q^n-1)/(q-1)$ points. Not all of these give
rise to examples, however. The case $L_3(3) < A_{13} < \Omega_{11}(13)$
can be analysed using the classification of maximal subgroups of orthogonal
groups in $11$ dimensions in \cite{BHRD}, where we find that $\Omega_{11}(13)$ contains
two classes of $L_3(3){:}2$, so that $L_3(3)$ embeds in both $A_{13}$ and
$L_3(3){:}2$. Similarly, the cases
$L_4(2) < A_{15} < \Omega_{13}(\ell)$ for $\ell=3,5$ are described in
\cite{AKS}. There is one class of $A_8$, and two classes of $S_{15}$, 
in $\Omega_{13}(3)$, so $A_8$ is not second maximal in this case.
There is one class of $A_{15}$, and two classes of $S_8$, in $\Omega_{13}(5)$,
so $A_8$ is contained in three maximal subgroups in this case.

\section{Infinite series of examples}
If $p$ is a prime bigger than $7$, then there is an
embedding  $L_3(2)<A_7<\Omega_6^\varepsilon(p)$, 
where $\varepsilon=+$ just when $p$ is a quadratic residue modulo $7$.
For simplicity, restrict to the case $\varepsilon=+$.
We read off the following properties from \cite[Table 8.9]{BHRD}.
The number of classes of 
$A_7$ is at least $2$, and is exactly $2$ when $p\equiv 3\bmod 4$.
The same is true for $L_3(2)$. In this case, the centre of $\Omega_6^+(p)$ is trivial,
and the outer automorphism group has order $4$, consisting of a diagonal automorphism
$\delta$, a graph automorphism $\gamma$, and their product $\delta\gamma$.
Now $A_7$ is normalized by $\gamma$ in all cases, while $L_3(2)$ is normalized
by $\gamma$ provided $p\equiv \pm1\bmod 8$, and by $\delta\gamma$ otherwise.
Thus we must restrict to the case $p\equiv 7\bmod 8$, and $p\equiv 1,2,4\bmod 7$,
that is $p\equiv 15,23,39\bmod 56$. In these cases, the group 
$\Omega_6^+(p).\langle\gamma\rangle=SO_6^+(p)$ contains two classes of $S_7$,
and two classes of $L_3(2){:}2$. The automorphism $\delta$ swaps the two classes of
$S_7$, and swaps the two classes of $L_3(2){:}2$.

%Thus we have obtained an infinite series of counterexamples to
%Aschbacher's question. 
\begin{theorem}
Let $p$ be a prime, %$p>7$, 
and suppose that $p\equiv 15,23,39\bmod 56$. Let $G$ be the
simple group $\Omega_6^+(p)\cong PSL_4(p)$. Then
$G$ contains subgroups $L_3(2)<A_7$, both normalized by the tranpose-inverse
automorphism of $L_4(p)$, to $L_3(2){:}2$ and $S_7$ respectively.
In particular, every such $L_3(2)$ lies in exactly two copies of $A_7$,
and these two copies of $A_7$ are $G$-conjugate.
\end{theorem}
\begin{proof}
Consider a pair of subgroups $L_3(2)<A_7$ of $G$, and adjoin $\alpha\gamma$,
where $\alpha$ is an inner automorphism of $\Omega_6^+(p)$, to extend
$A_7$ to $S_7$. This swaps the two classes of $L_3(2)$ in $A_7$. But we can also
adjoin $\beta\gamma$, where $\beta$ is another inner automorphism, to normalize
$L_3(2)$ to $L_3(2){:}2$. Hence there is an inner automorphism of the form
$\alpha\gamma\beta\gamma$, that conjugates an $L_3(2)$ of one class in $A_7$,
to an $L_3(2)$ of the other class. The same argument with the roles of $L_3(2)$
and $A_7$ reversed shows that the two copies of $A_7$ in which $L_3(2)$ lies
are conjugate in $\Omega_6^+(p)$.
%These facts can be read off from \cite[Table 8.9]{BHRD}.
\end{proof}
As a consequence, we have an infinite series of groups $PSL_4(p)$, for $p$ any prime
with $p\equiv
%1,9,
15,23,39 %,43
\bmod 56$, for which Aschbacher's question has a negative answer.
There is a similar infinite series of groups $PSU_4(p)$, for $p\equiv1\bmod8$ and
$p\equiv 3,5,6\bmod7$, that is
$p\equiv %13,
17,33,41 %,%47,55
\bmod 56$.
This can be read off in a similar way from \cite[Table 8.11]{BHRD}.
\begin{theorem}
Let $p$ be a prime, %$p>7$, 
and suppose that $p\equiv 17,33,41\bmod 56$.
Let $G$ be the simple groups $\Omega_6^-(p)\cong PSU_4(p)$. Then $G$ contains
subgroups $L_3(2)<A_7$, both normalized by the field automorphism of $U_4(p)$,
to $L_3(2){:}2$ and $S_7$ respectively. In particular, every such $L_3(2)$
lies in exactly two copies of $A_7$, and these two copies of $A_7$
are conjugate in $G$. 
\end{theorem}

There is an analogous embedding $L_2(11)<A_{11}$,
which one might think gives similar series of examples in
$\Omega_{10}^\varepsilon(p)$ for certain $p$. 
However, $L_2(11)$ is not maximal in $A_{11}$, so this fails.
%The proof is essentially the same
%as above.
%\begin{theorem}
%Let $p$ be a prime, $\varepsilon=\pm$, and suppose that 
%$\varepsilon p\equiv -1\bmod 12$, and $\varepsilon p$ is a square modulo $11$,
%that is $\varepsilon p\equiv 1,3,4,5,9\bmod 11$. 
%Let $G=\Omega_{10}^\varepsilon(p)$, and $H\cong L_2(11)$ a subgroup of $G$.
%Then $H$ lies in exactly two maximal subgroups of $G$, isomorphic to $A_{11}$,
%and conjugate in $G$.
%\end{theorem}

\section{More special examples}
As we have just seen,
the embedding $L_3(2)<A_7$ behaves differently in characteristic $7$ 
(the \emph{special} case) from
other characteristics (the \emph{generic} case). 
More generally, the embedding $L_n(q)<A_d$, where
$d=(q^n-1)/(q-1)$, behaves differently in 
the special case (characteristic dividing $d$),
compared to the generic case (characteristic prime to $d$). 

The special case is easiest to analyse when $d$ is itself prime.
In this case, $n$ is necessarily prime, but $q$ need not be prime.
This includes all Mersenne primes except $3$, and others such as $(3^3-1)/(3-1)=13$
and
$(5^3-1)/(5-1)=31$, for example.
%$(8^3-1)/(8-1)=73$,
%$(101^3-1)/(101-1)=10303$, $(5^{11}-1)/(5-1)=12207031$,
%$(17^7-1)/(17-1)=25646167$, $(73^7-1)/(73-1)=153436090543$ and many others.
%In this case, $n$ is necessarily prime, but $q$ need not be.
We then have embeddings $L_n(q) < A_d < \Omega_{d-2}(d)$. The Singer cycles
in $L_n(q)$ are represented as $d$-cycles in
$A_d$, and as regular unipotent elements in $\Omega_{d-2}(d)$.
Now there is a unique class of regular unipotent elements in $SO_m(d)$ for all odd $m$,
and these elements have order $d$ provided $m\le d$. The class splits into two
classes in $\Omega_m(d)$, and these classes are rational if $m\equiv \pm1\bmod 8$,
and irrational otherwise.

Since $d$ is prime, the $d$-cycles in $S_d$ %always 
split into two irrational classes in $A_d$ (by Sylow's Theorem).
The $d$-cycles are conjugate in $A_d$ to their inverses just when $d\equiv 1\pmod 4$.
Since the regular unipotent elements have unipotent centralizer, it follows that
they are conjugate in $\Omega_{d-2}(d)$ to 
their inverses if and only if either $d\equiv 1\bmod 4$ or $d-2\equiv \pm1 \bmod 8$,
that is $d\equiv 1,3,5\bmod 8$.
Now the Singer cycles in $L_n(q)$ are inverted by the transpose-inverse
automorphism, and we want this automorphism to be realised by an element
of $SO_{d-2}(d)\setminus \Omega_{d-2}(d)$. This happens if and only if
$d\equiv 7\bmod 8$.

\begin{theorem}
If $q$ is a prime power, and $d:=(q^n-1)/(q-1)$ is prime, with $d\equiv 7\bmod 8$,
let $H=P\Gamma L_n(q)$, $M=A_d$ and $G=\Omega_{d-2}(d)$. Then $H<M<G$, 
and $H$ and $M$ are
unique up to conjugacy in $G$. Hence $H$ and $M$ extend to $H.2$ and $M.2$ in $G.2$,
and $H$ is contained in exactly two $G$-conjugates of $M$. 
\end{theorem}

The condition $d\equiv 7\bmod 8$ is satisfied by all Mersenne primes
(the case $q=2$), except $3$, but not
by all primes of the form $(q^n-1)/(q-1)$. The condition can be re-written as a
condition on the values of $q$ and $n$ modulo $8$.
\begin{lemma}
If $d=(q^n-1)/(q-1)$, then the condition $d\cong 7\bmod 8$ is equivalent
to the condition that, either
%Either %$q=2$, or
\begin{itemize}
\item $q=2$ and $n>2$, or
\item $q\equiv 1\bmod 8$ and $n\equiv 7\bmod 8$, or
\item $q\equiv 5\bmod 8$ and $n\equiv 3\bmod 8$.
\end{itemize}
\end{lemma}

Only finitely many primes $d$ of the form $(q^n-1)/(q-1)$ are known,
but it is conjectured that there are infinitely many, including
infinitely many Mersenne primes $2^n-1$. Currently just $50$ Mersenne primes
are known, giving rise to examples with $H$ isomorphic to $L_3(2)$,
$L_5(2)$, $L_7(2)$, $L_{13}(2)$, \ldots, $L_{77232917}(2)$.
Less effort has been expended on finding primes for larger values of $q$,
but examples for $q=5$ 
and $n\equiv 3\bmod 8$
occur when $n=3$, $11$, $3407$, $16519$, $201359$ and $1888279$
(see A004061 in the On-line Encyclopedia of Integer Sequences \cite{OEIS}).
I could find no examples with $q=9$ or $q=13$, but using GAP \cite{GAP},
one can easily find the examples $n=7$, $47$ and $71$ for $q=17$
and $n\equiv 7\bmod 8$.

One can also search for examples by fixing $n$ rather than $q$.
For $n=3$, examples with $q\equiv 5\bmod 8$ and $d$ prime include
$q=5, 101, 173, 293, 677, 701, 773$. A search with $n=7$ turns up the examples 
$q=17, 73,89,353,1297,1409,1489,1609$, $1753$, $2609$, $2753,3673,4049,4409$, etc.,
and similarly for $n=11$, we can take $q=53,229,389, 709, 1213, 2029, 5581,
5669, 5813, 5861, 7229
$.
For $n=19$, there are examples for $q=181,277,389,509, 797, 1693, 1709$, etc.
For $n=23$, $q=113$, $257$, $857$, $1801$; for $n=31$, $q=241$, and so on.
%Examples with $q=5$ are $n=3, 11, 3407, 16519, 201359, 1888279$.
%(See A004061 in OEIS.)
%I could find no examples with $q=9$ or $q=13$.
%Examples with $q=17$ (found using GAP) are $n=7,47,71$.

In particular, examples of negative answers to Aschbacher's question
arise in the cases of $L_3(5)$, $L_3(101)$, $L_{11}(5)$, $L_7(17)$,
and $L_7(73)$. An extremely large example arises from the embedding of
$L_{77232917}(2)$ in $A_d$ and $\Omega_{d-2}(d)$, where 
$d=2^{77232917}-1$ is the largest currently known Mersenne prime.

Note that our analysis shows that the case $L_3(3)<A_{13}<\Omega_{11}(13)$ 
does not provide an example. This can also been seen from the list
of maximal subgroups of $\Omega_{11}(q)$ given in %appears in
\cite[Table 8.75]{BHRD}. %, from which one can see that this case
%does not give an example. 
%
%Mersenne primes arise in the cases $L_3(2)$, $L_5(2)$, $L_7(2)$, $L_{13}(2)$, and so on.
%Of course, only $50$ Mersenne primes are known, and it is not known if there are
%infinitely many. Nevertheless, the embedding of $L_{77232917}(2)$ in $A_d$
%and $\Omega_{d-2}(d)$, where $d=2^{77232917}-1$, gives an extremely large
%example where Aschbacher's question has a negative answer.

\section{More generic examples}
%\section{Generalisations?}
%More generally, 
As we have seen, for all $n\ge 3$ and for all $q$,
the simple groups $L_n(q)$ %for all $n\ge 3$ 
have two inequivalent
permutation representations on $d:=(q^n-1)/(q-1)$ points,
 and hence we obtain
two inequivalent embeddings in $A_d$. 
%and thence into $\Omega_{d-2}(\ell)$ for primes $\ell$
%dividing $d$, 
In the generic case, when $\ell$ is a prime
not dividing $d$, the alternating group $A_d$
embeds irreducibly into $\Omega_{d-1}(\ell)$. %for all other primes $\ell$.
However, the conditions on $n,q,\ell$ for this to give rise to a negative
answer to Aschbacher's question, are subtle and complicated, as we already
saw for the smallest case, $n=3,q=2$.

The next smallest case is $n=3,q=3$. To analyse this case, that is,
%Here we have an embedding $L_3(3)<A_{13}<\Omega_{11}(13)$.
%Quoting results from \cite[Table 8.75]{BHRD},
%we find that this does not in fact give a negative answer.
%For 
the embedding $L_3(3)<A_{13}<P\Omega_{12}^\pm(p)$, we may use 
the information on maximal subgroups of
$\Omega_{12}^\pm(p)$ provided in \cite[Tables 8.83 and 8.85]{BHRD}. 
It follows from these tables that
there are no examples here.

The next smallest case is $n=4,q=2$, and the embedding of $L_4(2)\cong A_8$
into $A_{15}$ and thence into
orthogonal groups in dimensions $13$ and $14$. The maximal subgroups of these
orthogonal
groups have been determined by Anna Schroeder, in her St Andrews PhD thesis \cite{AKS}.
In particular, the embeddings 
into $\Omega_{13}(3)$ and $\Omega_{13}(5)$ do not give examples. 
In the dimension $14$ case,
however, it seems that there is a crucial error at exactly the point that
interests us here: maximal subgroups $S_8$ are eliminated from the lists 
of maximal subgroups of $P\Omega_{14}^\pm(p)$ by the
assertion, in \cite[Propn. 6.4.17(iv)]{AKS}, 
that $S_8\le S_{15}$, which is manifestly false for this embedding.
Indeed, Propositions 6.4.4 and 6.4.5 in \cite{AKS} give the true picture,
and show that $S_8$ is indeed a maximal subgroup of $SO_{14}^\varepsilon(p)$
for suitable congruences of $\varepsilon$ and $p$.
%It seems likely that for $p$ satisfying suitable congruences, an infinite
%series of examples can be found of the type $A_8<A_{15}<\Omega_{14}^\pm(p)$,
%and that the proof will be much the same as for the case $L_3(2)<A_{7}<\Omega_6^\pm(p)$.
%For general $p$,
%this depends on answers to delicate questions of which outer automorphism
%of $\Omega_{14}^\pm(p)$ normalizes this subgroup $A_8$, and which
%normalizes an $A_{15}$ containing it. 
In the cases when the outer
automorphism group of $\Omega_{14}^\varepsilon(p)$ is just $2^2$, 
the calculations are quite straightforward. These are the cases 
when $\varepsilon p\equiv 3\bmod 4$.

\begin{theorem}
Let $p\equiv 19,23,31,47\bmod 60$, and let $G=\Omega_{14}^+(p)$. Let $H\cong A_8$
be a subgroup of $G$ acting irreducibly in the $14$-dimensional representation.
Then $H$ is contained in exactly two maximal subgroups of $G$, both isomorphic to
$A_{15}$, and conjugate to each other in $G$. 
\end{theorem}
\begin{proof}
Indeed, it is shown in \cite{AKS} that for $p\equiv 19,23,31,47\bmod 60$,
there are two conjugacy classes of subgroups $S_{15}$, maximal in $SO_{14}^+(p)$, 
and swapped by the diagonal automorphism $\delta$.
Moreover, it is shown that the intersection of $S_{15}$ 
with $\Omega_{14}^+(p)$ is $A_{15}$. %There are two conjugacy classes, swapped by
%the diagonal automorphism $\delta$. 
Now the same argument applies to the
group $S_8$, acting irreducibly in the $14$-dimensional representation. Since
for this embedding, $S_8$ does not lie in $S_{15}$, it follows that
$S_8$ is maximal in $SO_{14}^+(p)$ in these cases.
\end{proof}

Exactly the same argument applies to the cases $p\equiv 13,29,37,41\bmod 60$
in $\Omega_{14}^-(p)$. It is possible that analogous examples also exist when
$\varepsilon p\equiv 1\bmod 4$, but in this case the outer automorphism group 
is $D_8$, and there are four classes each of $S_8$ and $S_{15}$, so the
situation is more complicated.

\section{Unbounded rank}
Generalizing to the representations of $L_n(2)$ of dimension $d-1$, where
$d=2^n-1$, we find that for $n$ even these representations
extend to emebeddings of $L_n(2){:}2$ in $SO_d(p)$
for all $p$, while for $n$ odd this happens only when the field of
order $p$ contains square roots of $2$, that is, when $p\equiv \pm 1\bmod 8$.

So, for example, the embeddings $L_5(2) < A_{31} < \Omega_{30}^\varepsilon(p)$
provide examples whenever all of the following conditions are satisfied:
\begin{itemize}
\item $\varepsilon p\equiv 3 \bmod 4$, 
\item $p\equiv \pm 1\bmod 8$, and
\item $p\equiv 1,2,4,8,16\bmod 31$. 
\end{itemize}
That is to say, for $\varepsilon=+$ we require
$p\equiv 39,47,63,95,159\bmod 248$, while for $\varepsilon=-$ we require
$p\equiv 1,33,97,225,233\bmod248$. 

For the purpose of demonstrating that examples exist of arbitrarily large rank
and/or characteristic,
it is sufficient to restrict to the case when $\varepsilon=-$,
 and further to the case when $p\equiv 1\bmod 4(d-1)$.
In this case, the embedding of $L_n(2)$ into $A_d$ and thence into
$\Omega_{d-1}^-(p)$ gives an example of a negative answer to Aschbacher's
question. Of course, there are many other examples.

\begin{theorem}
Let $p$ be a prime, and $\varepsilon=\pm$, such that $\varepsilon p\equiv 3\bmod 4$.
%and $p\equiv\pm1\bmod 8$. 
Let $n\ge 3$, and suppose that $p$ is a square modulo $d:=2^n-1$.
If $n$ is odd, suppose also that $p\equiv \pm1\bmod 8$.
Let $G=\Omega_{d-1}^\varepsilon(p)$, and $H\cong L_n(2)$ a subgroup of $G$.
Then $H$ is contained in exactly two maximal subgroups of $G$, which are
isomorphic to $A_d$ and conjugate to each other.
\end{theorem}
\begin{proof}
The above conditions ensure that $G$ has outer automorphism group of order $4$,
and that both $L_n(2){:}2$ and $S_d$ embed in $SO_{d-1}^\varepsilon(p)$ but
not in $\Omega_{d-1}^\varepsilon(p)$. Hence we have the same configuration as in
all the other examples above.
\end{proof}

We have now shown that there is no bound on the Lie rank of those $G$
for which Aschbacher's question has a negative answer.
In these examples, there are two conjugacy classes of $L_n(2)$ in 
$\Omega_{d-1}^\varepsilon(p)$, and two conjugacy classes of $A_d$, interchanged by
the diagonal automorphism. If instead $\varepsilon p\equiv1\bmod 4$, then there
are four classes of each, and the outer automorphism group of 
$\Omega_{d-1}^\varepsilon(p)$ is $D_8$. In \cite{BHRD}, two of the reflections
in $D_8$ are described as graph automorphisms $\gamma$, and the other two as
$\delta\gamma$, but unfortunately the two conjugates of $\gamma$ are not
distinguished from each other. For any particular choice of $\gamma$, two of the
four classes of $A_d$ extend to $S_d$, and the other two classes are interchanged.

%The previous sections suggest that it will be fruitful to look at the embeddings
%of $L_n(2)$ in $A_{d}$, where $d=2^n-1$. In the cases where $d$ is prime (that is,
%a Mersenne prime) we then get an embedding in $\Omega_{d-2}(d)$. Since $d\equiv 3\bmod 4$,
%involutions inverting the elements of order $d$ in $S_d$ are odd permutations.
%On the other hand, these elements are regular unipotent elements in $SO_{d-2}(d)$,
%so are not centralized by any semisimple elements. 
%Moreover, since $d\equiv 7\bmod 8$ we know that the involutions in $S_d$ that invert
%the $d$-cycles lie in $SO_{d-2}(d)\setminus \Omega_{d-2}(d)$.
%Hence every element in
%$SO_{d-2}(d)$ that inverts such a regular unipotent element lies outside
%$\Omega_{d-2}(d)$. In particular, the transpose-inverse automorphism of $L_n(2)$
%is realised by an element of $SO_{d-2}(d)\setminus \Omega_{d-2}(d)$.
%Hence we have the same set-up as in all the previous examples, and hen

\section{Other classical groups}
%\section{Symplectic groups}
So far, all our examples with $G$ a classical group have occurred when $G$ is
in fact orthogonal. There is no bound on the characteristic, and there is
no bound on the rank. All three families of orthogonal groups (plus type,
minus type, and odd dimension) occur. It would be interesting to know if
the other classical groups, linear, unitary or symplectic, can occur.

Of course, the isomorphisms $L_4(p)\cong \Omega_6^+(p)$ and
$U_4(p)\cong \Omega^-(p)$ imply the existence of examples in linear and
unitary groups, but do examples exist in linear and unitary groups of
larger dimension? 
So far, I have not found any examples. The large outer automorphism groups
in these cases make the analysis very delicate.
%the linear and unitary groups too complicated to deal with.
There are potential examples of the form $L_3(4)<U_4(3)<L_6(p)$, but there
are three classes of each of $L_3(4)$ and $U_4(3)$, and the embeddings between them
are not given explicitly in \cite{BHRD}. Hence one needs extra detailed information
to resolve these cases. It seems likely, however, that this configuration
does not give any examples.

In the case of symplectic groups, over fields of odd prime order,
the outer automorphism group has order $2$, which is the ideal
situation for us. If one looks
%Looking 
through the tables of maximal subgroups of symplectic groups in dimensions up to $12$
given
in \cite{BHRD}, one finds, %no other examples 
besides the case
$L_3(2)<A_7<S_4(7)$ already discussed, just one series of
potential examples, given by the embeddings
$A_5 < L_2(p) < S_6(p)$ for $p$ a prime, $p\equiv \pm11,\pm19\bmod 40$. 
However, in this case the embedding of $2\udot A_5$ in $Sp_6(p)$ also goes via 
the tensor product $Sp_2(p)\circ GO_3(p)$, so this $A_5$ lies in more than
two maximal subgroups of $Sp_6(p)$.
%\begin{theorem}
%Let $p$ be a prime, $p\equiv \pm11,\pm19\bmod 40$, and let $G\equiv S_6(p)$.
%Let $H\cong A_5$ be a subgroup of $G$, acting absolutely irreducibly
%on the natural module. Then $H$ is unique up to conjugacy, and is contained in
%exactly two maximal subgroups of $G$, isomorphic to $SL_2(p)$, and conjugate in $G$.
%\end{theorem}
%\begin{proof}
%This follows from \cite[Table 8.29]{BHRD}.
%\end{proof}

%Moreover, 
On the other hand, Anna Schroeder's PhD thesis \cite{AKS}
contains the lists of maximal subgroups of $S_{14}(q)$ and their 
automorphism groups. There one finds two more potential infinite series of examples,
given by the embeddings $J_2<S_6(p)<S_{14}(p)$ and $L_2(13)<S_6(p)<S_{14}(p)$
for suitable primes $p$. The relevant congruences are $p\equiv \pm11,\pm19\bmod 40$
for $J_2$, and $p\equiv\pm3,\pm27,\pm29,\pm35,\pm43,\pm51\bmod 104$
for $L_2(13)$. It is straightforward to check, in the same way
as before, that these
do indeed give examples of negative answers to Aschbacher's question.

\begin{theorem}
Let $p\equiv \pm11,\pm19\bmod 40$, and let $G=S_{14}(p)$. Let $H\cong \Ja_2$ be
a subgroup of $G$. Then $H$ is contained in exactly two maximal subgroups of $G$,
both isomorphic to $S_6(p)$, and conjugate to each other in $G$.
\end{theorem}
\begin{theorem}
Let $p\equiv \pm3,\pm27,\pm29,\pm35,\pm43,\pm51\bmod 104$, and let $G=S_{14}(p)$.
Let $H\cong L_2(13)$ be a subgroup of $G$ contained in $M\cong S_6(p)$. Then
$H$ is contained in exactly two maximal subgroups of $G$, both conjugate to $M$.
\end{theorem}

%\section{Linear and unitary groups}
\section{Further remarks} 
Far-reaching as the above examples are,
they have little, if any, impact on
%It is not clear to me what impact, if any, these examples have on
Aschbacher's programme. This is because they all occur in sporadic or
classical groups, whereas Aschbacher is only proposing to use this
approach for exceptional groups of Lie type.
Our examples therefore merely show that his question is still too
broad, and that the question needs to be restricted to a smaller class
of groups than the class of almost simple groups.

%Examples of such classes that our examples do not touch are
%\begin{itemize}
%\item almost simple exceptional groups;
%\item almost simple groups that are not simple.
%\end{itemize}

%For the applications that Aschbacher has in mind, he requires a positive
%answer only to some reasonable special case of \cite[Question 8.1]{Asch}. 
%Thus 
%the existence of three simple counterexamples is of no consequence for the applications.
%Our examples show that the question is still much too broad. 
%For the sake of this discussion, let us say that a group $G$ has the MONC 
%(maximal-overgroups-not-conjugate) property
%if \cite[Question 8.1]{Asch} has a positive answer for $G$. 
%The obvious meta-question
%is now, in what reasonable class (or classes)
%of groups does every group have the MONC property?

%We now know that not every simple group has the MONC property, but our counterexamples
%are either sporadic or classical.
%Within the classical family, we have infinite families of examples in both 
%orthognal and symplectic groups, but not in linear or unitary groups
%except in dimension $4$, when these groups can also be interpreted as
%orthogonal groups in dimension $6$.
%(specifically, orthogonal, in characteristic $2$).
%On the other hand,
%the meta-question of what reasonable class of groups has a positive answer
%
%The discussion in \cite{Asch} concentrates on the case when
%$G$ is an exceptional group of Lie type,
%but I know of no counterexamples in this case. 
%Thus it would be
%reasonable to ask specifically whether every simple, or almost simple,
%exceptional group of Lie type has the MONC property.
The maximal subgroups are known completely for five of the ten families
of exceptional groups of Lie type,
and, of the remaining five, $E_8$ seems least likely to be a source of examples, since it 
admits neither diagonal nor graph automorphisms. Similarly, $F_4$ admits
no diagonal automorphisms, and admits a graph automorphism only in
characteristic $2$. Probably the most promising places to look for examples are
in $E_6$ with a graph automorphism, and in $E_7$ with a diagonal
automorphism. 

On the other hand, it is entirely conceivable that no examples exist in
the exceptional groups of Lie type. 

%Even in this case, however, a finite number of counterexamples would 
%be of little consequence for the purposes of Aschbacher's proposed application.
%so it is worth providing a counterexample in this case also. Let us turn therefore
%to \cite[p.~191]{Atlas} and consider $G={}^2E_6(2)$, and $H\cong \Omega_7(3)$.
%A proof that the Atlas lists of maximal subgroups of $G$ and of its
%extensions by outer automorphisms are complete can be found in \cite{max2E6}.
%It follows immediately that the overgroups of $H$ in $G$ are precisely two 
%copies of $Fi_{22}$, that are conjugate in $G$.
%Thus one might instead ask whether all but finitely many finite simple groups
%have the MONC property. 
%Likewise the question of whether there are finitely many or infinitely
%many examples remains open. If there are infinitely many, 
%one can also ask whether there are
%examples of arbitrarily large Lie rank, and/or arbitrarily large field size.
%
%So far, a targeted search for an infinite family of simple groups that
%do not have the MONC property has failed. The failure of the MONC property
%would appear to be a rare phenomenon. Just how rare it is remains an open question.

\section*{Acknowledgements}
I thank the Martin-Luther-Universit\"at Halle-Wittenberg
for a visiting professorship, during which the first version
of this paper was written, and especially
Professor Rebecca Waldecker, for inviting me, and arranging everything.
%, and for
%fruitful mathematical discussions. 
I thank Rebecca Waldecker, Imke Toborg and Kay Magaard
for fruitful discussions.

\end{document}